\documentclass[12pt]{amsart}

\usepackage[a4paper,hmargin=3.5cm,vmargin=4cm]{geometry}
\usepackage{amsfonts,amssymb,amscd,amstext}
\usepackage{graphicx}
\usepackage[dvips]{epsfig}

\usepackage{fancyhdr}
\usepackage{color}

\input xy
\xyoption{all}

\usepackage{times}
\usepackage{enumerate}
\usepackage{titlesec}
\usepackage{mathrsfs}

\titleformat{\section}
{\filcenter\bfseries\large} {\thesection{.}}{0.2cm}{} 
\titleformat{\subsection}[runin]
{\bfseries} {\thesubsection{.}}{0.15cm}{}[.]
\titleformat{\subsubsection}[runin]
{\em}{\thesubsubsection{.}}{0.15cm}{}[.]

\usepackage[up,bf]{caption}

\newtheorem{theorem}{Theorem}[section]

\newtheorem{lemma}[theorem]{Lemma}

\theoremstyle{definition}
\newtheorem{definition}[theorem]{Definition}

\numberwithin{equation}{section}
\numberwithin{figure}{section}

\def\Ccal{\mathcal{C}}

\def\Lcal{\mathcal{L}}

\def\Ocal{\mathcal{O}}

\def\c{\mathbb{C}}
\def\cn{\mathbb{C}^n}
\def\z{\mathbb{Z}}
\def\d{\mathbb{D}}
\def\cd{\overline{{D}}}
\def\b{\mathbb{B}}

\def\r{\mathbb{R}}
\def\n{\mathbb{N}}

\def\z{\mathbb{Z}}

\def\dist{\mathrm{dist}}

\newcommand\di{\partial}
\newcommand\dibar{\overline\partial}

\begin{document}

\fancyhead[LO]{Proper holomorphic curves attached to domains}
\fancyhead[RE]{B.\ Drinovec Drnov\v sek \& M.\ Slapar}
\fancyhead[RO,LE]{\thepage}

\thispagestyle{empty}


\begin{center}
{\bf\Large Proper holomorphic curves attached to domains}

\vspace*{0.3cm}

{\large\bf Barbara Drinovec Drnov\v sek \& Marko Slapar}
\end{center}


\vspace*{0.5cm}

\begin{quote}
{\small
\noindent {\bf Abstract}\hspace*{0.1cm}   
Let $D\Subset \c^n$ be a domain with smooth boundary, of finite 1-type at a point $p\in bD$ and such that $\cd $ has a basis of Stein Runge neighborhoods. Assume that there exists an analytic disc which intersects $\cd$ 
exactly at $p$.
We construct proper holomorphic maps from any open Riemann surface $S$ to  $\cn$ which are attached to $\cd$ 
exactly  at $p$.
\vspace*{0.1cm}

\noindent{\bf Keywords}\hspace*{0.1cm} weakly pseudoconvex domain, holomorphic curve.

\vspace*{0.1cm}


\noindent{\bf MSC (2010)} \hspace*{0.1cm} Primary: 32C25; 
Secondary: 32H02, 32H35.}
\end{quote}



\section{Introduction}
\label{sec:intro}

We study sufficient conditions for the existence of a proper 
holomorphic curve attached to the boundary of a domain, which is parametrized by 
an {\em open Riemann surface} or by {\em a finite bordered Riemann surface}, i.e.,
a one dimen\-sional complex manifold with compact closure
$\bar S=S\cup bS$ whose boundary $bS$ consists of finitely many closed 
Jordan curves.

We will denote  by $\d$ the open unit disc in $\c$.
Let $D\subset \c^n$ be a bounded domain with smooth boundary, let $p \in bD$ and
let $\rho$ be a smooth defining function for $bD$ near $p$. 
Assume that $f\colon \d\to \c^n$  
is a holomorphic map  such that $f(0)=p$. 
If $\rho\circ f$ has a zero of finite order at $0$ and
$f(r\d\setminus\{0\})\cap \overline D=\emptyset$ for some $r>0$
then we say that \emph{$f$ has
a finite order of contact with $\overline D$ at $p$}. 
The definition does not depend on the choice of the defining function.
It is local and it extends to holomorphic maps from Riemann surfaces.


\begin{theorem}
\label{thm1}
Let $D\subset \c^n$ be a bounded domain with smooth boundary, $p\in bD$,
and assume that $\cd $ has a basis of Stein Runge neighborhoods.

\begin{itemize}
\item[(1)]  
Given a bordered Riemann surface $S$, $s\in S$, and a continuous map 
$f\colon \bar S\to \c^n$, holomorphic on $S$, with $f(s)=p$, having a 
finite order of contact with $\overline D$ at $p$, and $f(S\setminus\{s\})\cap \cd  =\emptyset$,
 there exists a proper holomorphic map 
$g\colon S\to\c^n$ with $g(s)=p$ and $g (S\setminus\{s\})\cap \cd  =\emptyset$.
\item[(2)]
Given an open Riemann surface $S$, a compact $\Ocal (S)$-convex subset $K\subset S$ with the 
nonempty interior $\mathring K$,
$s\in \mathring K$ and a holomorphic map 
$f\colon K\to \c^n$ with $f(s)=p$, having a finite order of contact with $\overline D$ at $p$, and 
$f(K\setminus\{s\})\cap \cd  =\emptyset$, 
there exists a proper holomorphic map 
$g\colon S\to\c^n$ with $g(s)=p$ and $g (S\setminus\{s\})\cap \cd  =\emptyset$.
\end{itemize}

In particular, if $D$ is totally pseudoconvex and of finite 1-type 
at a point $p\in bD$,
then there exist a proper holomorphic map 
$g\colon\d\to\c^n$ with $g(0)=p$  and $g (\d\setminus\{0\})\cap \cd  =\emptyset$ and a proper holomorphic map 
$g\colon\c\to\c^n$ with $g(0)=p$ and $g (\c\setminus\{0\})\cap \cd  =\emptyset$.
\end{theorem}

It follows by Remmert's proper mapping theorem that the image $g(S)$ is 
an analytic subvariety of $\c^n$,
in particular, if $n=2$, it is a global support hypersurface for $D$ at $p$:
We call an analytic
hypersurface $M$ \emph{a local support hypersurface} for $D$ at $p$ if there is a 
neighborhood $U$ of $p$ in $\c^n$ such that  $M\cap U\cap \cd =\{p\}$. If $M\cap \cd =\{p\}$ we call $M$ \emph{a global support hypersurface}.
The domain $D$ is called \emph{totally pseudoconvex} at $p\in bD$ if there is a  local 
support hypersurface for $D$ at $p$ \cite{BehnkeThullen1970,Range1978PJM}.
Let $\rho$ be a smooth defining function for $bD$, and let $p \in bD$. If the maximum order of vanishing of $\rho\circ f$ for all one dimensional complex curves $f\colon \d\to \c^n$, with $f(0)=p$  is finite, then we say that $bD$ is of \emph {finite  1-type} at $p$ \cite{DAngelo1993book}. Denote by $\Ocal (X)$ the algebra of all holomorphic functions on the complex manifold $X$, endowed with the compact-open topology. A compact set $K$ in $X$ is said to be $\Ocal (X)$-convex if for every point $x\in X\setminus K$ there exists $F\in \Ocal (X)$  with $|F(x)| > \sup_K |F|$.

We shall actually prove part (1) in Theorem \ref{thm1} for domains $D$ in $(n-1)$ convex 
complex manifolds $X$, see Theorem \ref{thmMainSec2}, and part (2) in Theorem \ref{thm1} for domains $D$ in Stein manifolds $X$ with the density property, see Theorem \ref{thmMainSec2.2}.
Furthermore, the map $g$ can be chosen an immersion, and if $\dim X\ge 3$ then $g$
can be chosen an embedding. Moreover, we are able to attach proper holomorphic curves 
to more general compact sets than smoothly bounded domains $\overline D$;
here we assume that the domain $D$ has $\Ccal^\infty$ smooth boundary.

We call the image of a proper holomorphic map from an open Riemann surface  \emph{a proper
holomorphic curve}.
Proper holomorphic discs in pseudoconvex domains through any given point were constructed 
in \cite{ForstnericGlobevnik1992} by Forstneri\v c and Globevnik. 
The results  evolved in various directions, for a thorough survey of recent results we refer to \cite{DrinovecForstneric2007DMJ}.

We say that a domain $D\subset \c^n$ is \emph{locally convexifiable} at
a point $p\in bD$ if there exists a biholomorphic change of coordinates near $p$ 
such that in the new coordinates near $p$ the boundary 
$bD$ is geometrically strictly convex with respect to the side on which $D$ lies.
In this case, there exists a local support hypersurface to $bD$ at $p$,
in particular, $D$ is totally pseudoconvex at $p$.
If $D\subset \c^n$ is a smoothly bounded strictly pseudoconvex domain then it is locally convexifiable at any boundary point by Narasimhan's lemma.  

Kohn-Nirenberg example \cite{KohnNirenberg1973MA} shows that local support hypersurfaces do not exist in general at weakly pseudoconvex points of pseudoconvex domains.
Recently, Diederich, Forn\ae ss and Wold \cite{DiederichFornaessWold2014JGA}
proved that for any bounded domain $D\subset \c^n$ which is locally
convexifiable and of finite 1-type near  $p\in bD$, and such that $\cd $ has a basis of Stein Runge neighborhoods, there 
exists an automorphism $\Psi$ of $\c^n$ such that $\Psi(p)$ is a global extreme point,
i.e.,  $\Psi(\overline D)\cap {b\b^n}=\Psi(p)$, where 
$\b^n$ denotes the open unit ball in $\c^n$. 
In particular, there exists a smooth support hypersurface for $D$ at $p$.

Kol\'a\v r \cite{Kolar1995MRL,Kolar2008MZ} constructed examples of smoothly bounded 
nonconvexifiable pseudoconvex domains with convex models, 
in particular, with local support hypersurfaces.
We provide an example of a bounded pseudoconvex domain $D\subset \c^2$ with smooth boundary, such that $\cd$  
has a Stein Runge neighborhood basis and a weakly pseudoconvex point $p\in bD$ of finite 1-type  such that the 
domain $D$ is not locally convexifiable at $p$ and 
that $D$ has a global support hypersurface at $p$, see Section \ref{example}.

\section{Control of the placement of the curve near a given point}
\label{sec:proof}

We will consider holomorphic curves attached to more general compact sets than closures of smoothly bounded domains:

\begin{definition}\label{defin}
Let $L\subset \c^n$ be a compact set, $p\in L$ and $f\colon \d\to \c^n$ a holomorphic map
such that $f(0)=p$. The map \emph{$f$ has a finite order of contact with $L$ at $p$}
if there are $C>0$, $r>0$ and $k\in\n$ such that
$\dist(f(z), L)\ge C|z|^k$ for all $|z|\le r$. Since the definition of a finite order of contact is local it extends to maps from Riemann surfaces to complex manifolds equipped with a Riemannian metric.
\end{definition}

For a smoothly bounded domain $D$,  $p \in bD$ and
$f\colon \d\to \c^n$ a holomorphic map  such that $f(0)=p$ 
the definition of a finite order of contact of $f$ with $\bar D$ at $p$ coincides with the
definition at the beginning  of this note: Since the signed distance $\rho$ is a defining function for $bD$,
the estimate in Definition \ref{defin} implies that $\rho\circ f$ has a zero of finite order at $0$. Conversely, choose a 
defining function $\rho$ of $bD$ near $p$. If $f\colon \d\to \c^n$ is a holomorphic map with $f(0)=p$ having a finite order of contact with $\overline D$ according to the definition in the Introduction, then $\rho\circ f$ has a zero of finite order at $0$ and
$f(r\d\setminus\{0\})\cap \overline D=\emptyset$ for some $r>0$.
Therefore, it holds that $(\rho\circ f)(z)=P_k(z,\bar z)+o(|z|^k)$, where $P_k$ is a real homogeneous
polynomial of degree $k$, and $P_k(z,\bar z)>0$ for $z\ne 0$. This implies that
there are $C>0$ and $r'>0$ small enough such that
$\dist(f(z),\overline D)\ge C|z|^k$ for all $|z|\le r'$.
Thus $f$ has a finite order of contact with $\overline D$ at $p$ according to 
Definition \ref{defin}.

If the map $f$ has a finite order of contact with compact set $L$ at $p$,
then the curve that intersects $L$ only in $p$ allows small perturbations, which, in a 
neighborhood of $p$, intersect $L$ only in $p$. 
The following lemma provides appropriate small perturbations of a holomorphic disc in a 
neighborhood of $p$.

\begin{lemma}
\label{lemmasmallperturb}
Let $L\subset \c^n$ be a compact subset, $p\in L$,
and $ f \colon\d\to\c^n$ a holomorphic map with $ f (0)=p$
and having a finite order of contact with $L$ at $p$.
Then there exist $r$, $0<r<1$, and an integer $k\ge 0$ 
such that for any $r'\in (r,1)$  there exist $\epsilon>0$ 
such that for any holomorphic map 
$ g \colon\d\to\c^n$ satisfying 
$ f ^{(m)}(0)= g ^{(m)}(0)$ for $m\in \{0,\ldots,k-1\}$ and 
$| f ( z )- g ( z )|<\epsilon $ for $| z |\le r'$ it holds that
 $g (r\overline\d\setminus\{0\})\cap L  =\emptyset$.
\end{lemma}

\begin{proof}
There are $C\in \r$, $r>0$ and $k\in\n$ such that
$\dist(f(z), L)\ge C|z|^k$ for all $|z|\le r$.
Therefore, for any holomorphic map $ g \colon\d\to\c^n$ satisfying
$| f ( z )- g ( z )|\le\frac C2 | z |^k$ for $| z |\le r$ it holds that
 $g (r\overline\d\setminus\{0\})\cap L  =\emptyset$. 

Choose any $r'\in (r,1)$.
Assume that $ g \colon\d\to\c^n$ is a holomorphic map  such that 
$ f ^{(m)}(0)= g ^{(m)}(0)$ for $m\in \{0,\ldots,k-1\}$.
By the Taylor expansion of $f-g$ around $0$ with the estimate of the remainder and by Cauchy formula we get
$$|f(z)-g(z)|\le C_1|z|^{k}\max\{|f(z)-g(z)|\colon z\in r'\overline\d\} \text{ for }z\in r\overline\d,$$
where the constant $C_1>0$ depends only on $r$ and $r'$.
We let $\epsilon=\frac{C}{2C_1}$. 
This proves the lemma.
\end{proof}

\section{Proper holomorphic curves in $(n-1)$-convex complex manifolds}

Part (1) in Theorem \ref{thm1} follows from the following theorem:

\begin{theorem}
\label{thmMainSec2}
Let $X$ be a complex manifold of $\dim X=n>1$ equipped with a Riemannian metric $d$ and let 
$L\subset X$ be a compact set.
Assume that for any neighborhood $W$ of $L$
there exists a smooth exhaustion function $\rho\colon X\to\r$ 
that is $(n-1)$-convex on $X_c=\{x\in X\colon \rho(x)>c\}$ for some $c\in \r$ such that $L\subset X\setminus \overline X_c\subset W$.  

Given $\epsilon>0$, a bordered Riemann surface $S$, $s\in S$, a compact subset $K\subset S$, 
a continuous map $f\colon \bar S\to X$, holomorphic on $S$ with $f(s)=p$, having
 a finite order of contact with $L$ at $p$ and such that
$f(S\setminus\{s\})\cap L  =\emptyset$, there exists a proper holomorphic map 
$g\colon S\to X$ with $g(s)=p$,  $g (S\setminus\{s\})\cap L  =\emptyset$ 
and $d(g(z),f(z))<\epsilon$ for $z\in K$.
\end{theorem}

Recall that a smooth function 
$\rho\colon X\to\r$ on a complex manifold of $\dim X=n>1$ is said to be {\em $q$-convex} 
on an open subset $U\subset X$ (in the sense of Andreotti-Grauert \cite{AndreottiGrauert1962BSMF},
\cite[def.\ 1.4, p.\ 263]{Grauert1994})
if its Levi form $i\di\dibar \rho$ 
has at most $q-1$ negative or zero eigenvalues at each point of $U$.
Note that $1$-convex functions are exactly strongly plurisubharmonic functions.
The manifold $X$ is {\em $q$-complete}, resp.\ {\em $q$-convex}, if it admits 
a smooth exhaustion function $\rho\colon X\to\r$ which is $q$-convex on $X$, %
resp.\ on $\{x\in X\colon \rho(x)>c\}$ for some $c\in \r$.
A 1-complete complex manifold is just a Stein manifold. 
If $X$ is a $(n-1)$-complete complex manifold with $(n-1)$-convex exhaustion 
function $\rho$, then we can take a sublevel set of $\rho$ for $L$ in the theorem.

The construction of the proper holomorphic map in Theorem \ref{thmMainSec2}
is inductive. The main addition to the previous constructions of proper holomorphic maps in
\cite{ForstnericGlobevnik1992,Globevnik2000IMJ,DrinovecForstneric2007DMJ}
is the control of the placement of the
disc near the point $p$ provided by Lemma \ref{lemmasmallperturb}.
Outside this neighborhood we control the 
placement of the curve in such a way that no intersection with $L$ occur. 

The following lemma will provide the main step 
in the inductive construction of a proper holomorphic map. 
It is a consequence of \cite[Lemma 4.2]{DrinovecForstneric2010AJM}
that assures that any map $f$ bellow is a core map of a spray of maps and 
\cite[Lemma 6.3]{DrinovecForstneric2007DMJ} that gives the new map $g$.
Main methods in the proof are the solution to a certain Riemann-Hilbert
problem which gives appropriate local corrections and the gluing of holomorphic sprays.

\begin{lemma}
\label{inductivestep}
Let $X$ be a complex manifold of dimension $n>1$   equipped with a Riemannian metric $d$ and let
$\rho\colon  X\to\r$ be a smooth exhaustion function 
which is $(n-1)$-convex on $\{x\in  X \colon\rho(x)>M_1\}$ for some $M_1\in \r$.
Let $S$ be a bordered Riemann surface, $s\in S$, let $U\Subset S$ be an open subset, let
$f\colon\overline S \to  X$ be  a continuous map, holomorphic on $S$,
and $M_1<M_2$, such that 
$f(z)\in \{x\in  X\colon\rho(x)\in (M_1,M_2)\}$ for all $z\in\overline S\setminus U$. 
Given $\varepsilon>0$, $M_3>M_2$, and an integer $k\ge 1$,
there exists a continuous map $g\colon\overline S\to  X$,  holomorphic on $S$,
satisfying the following properties:
\begin{itemize}
\item[(i)]   $g(z)\in \{x\in  X\colon\rho(x)\in (M_2,M_3)\}$ for $ z\in b S$,
\item[(ii)]  $g(z)\in \{x\in X\colon\rho(x)\in (M_1,M_3)\}$ for $z\in\overline S\setminus U$,
\item[(iii)] $d(f(z),g(z))<\varepsilon$ for $z\in \overline U$,
\item[(iv)]  the $k-1$ jets of $f$ and $g$ at $s$ are equal.
\end{itemize}
\end{lemma}

\begin{proof}[Proof of Theorem \ref{thmMainSec2}]
We choose local coordinates in a neighborhood $V$ of $p$ and an open neighborhood $U$ of $s$ in $S$ 
such that $U$ is biholomorphic to $\d$ and $f(U)\Subset V$.
By Lemma \ref{lemmasmallperturb} there is a neighborhood $U_0\Subset U$,
$\epsilon_1\in (0,\epsilon)$, and an integer $k>0$ such that 
for any holomorphic map $ g \colon S \to X$,
which has the same $k-1$ jet at $s$ as $f$, and
$d( f ( z ), g ( z ))<\epsilon_1 $ for $ z\in \overline U$ 
it holds that $g(\overline U)\subset V$ and
 $ g  (\overline U_0)\cap L  =\{p\}$.

Let $f_0=f$. There exist an exhaustion function $\rho\colon X\to \r$ and $M_0\in\r$ such that 
$\rho$ is $(n-1)$-convex on $\{x\in  X \colon\rho(x)>M_0\}$, and
\begin{equation}\label{r_0}
\max_{x\in \cd}\rho(x)<M_0<\rho (f_0(z)) \text{ for } z\in \overline S\setminus U_0.
\end{equation}

Let $M_{-2}=\max_{x\in\overline D}\rho(x)$. 
Choose an increasing sequence $\{M_i\}_{i\ge -2}$ converging to $\infty$
such that
$$M_{-2}<M_{-1}<M_0<\rho (f_0(z))<M_1 \text{ for }z\in \overline S\setminus U_0.$$

There is a decreasing sequence $\{\epsilon_i\}_{i\ge 1}$ converging to $0$ such that for $i\ge -2$,
\begin{equation}\label{rho}
\rho(x)\in [M_{i+1},M_{i+3}] \text{ and } d(x,y)<\epsilon_{i+3} \text{ then } \rho(x)>M_{i}.
\end{equation}

Choose a sequence of compact subsets $\{K_i\}_{i\ge 1}$ in $S$ such that $\cup_i K_i=S$.
We shall inductively construct a sequence of continuous maps $\{f_i\colon\overline S\to X\}_{i\ge 0}$, holomorphic on $S$, and
a sequence of open sets $\{U_i\}_{i\ge 0}$, $U_i\Subset U_{i+1}$, $U_0\cup K\subset U_1$, $\cup_i U_i=S$,  such that the following hold for all $i\in\z_+$:
\begin{enumerate}[\rm(i)]
		\item $\rho(f_i(z))\in (M_i,M_{i+1}) $ for $z\in bS$,
		\item $\rho(f_i(z))\in (M_i,M_{i+1})$ for $z\in \overline S\setminus U_{i+1}$,
		\item $\rho(f_{i}(z))\in (M_{i-1},M_{i+1})$ for $z\in \overline S\setminus U_{i}$,
		\item $d(f_{i+1}(z),f_{i}(z))<\frac{\epsilon_{i+1}}{2^{i+1}}$ for $z\in \overline U_{i+1}$,
		\item  the $k-1$ jets of $f_{i+1}$ and $f_{i}$ at $s$ are equal.
\end{enumerate}

Notice that by \eqref{r_0} the map $f_0$  meets conditions (i) and (iii) for $i=0$.
Choose  $U_1$, $K\cup U\Subset U_1$,  such that (ii) for $i=0$ is satisfied.
Let $j\in \n$ and assume that we have already constructed $f_0,\ldots,f_{j-1}$ and $U_0,\ldots,U_j$
that satisfy properties (i), (ii) and (iii) for $i=0,\ldots j-1$, and properties (iv) and (v) for  $i=0,\ldots j-2$.
We use Lemma \ref{inductivestep} to obtain a continuous map $f_j\colon\overline S\to X$, holomorphic on $S$,
which satisfies properties (i) and (iii) for $i=j$, and properties (iv) and (v) for $i=j-1$. Now we choose $U_{j+1}$, $U_j\cup K_j\Subset U_{j+1}$, such that $f_j$ satisfies (ii)
for $i=j$. This finishes the inductive construction.

It follows by (iv) that the sequence $\{f_n\}$ converges uniformly on compact sets in $S$
to a holomorphic map that will be denoted by $g$. 
For all $j\ge -2$ and $z\in \overline U_{j+3}$ we get by (iv) 
\begin{align}
\nonumber
d(g(z),f_{j+2}(z))&\le d(f(z)_{j+2},f_{j+3}(z))+d(f(z)_{j+3},f_{j+4}(z))+\cdots \\
\label{ocenaf}
&<\frac{\epsilon_{j+3}}{2^{j+3}}+\frac{\epsilon_{j+4}}{2^{j+4}}+\cdots <\epsilon_{j+3} .
\end{align}

The property (i) for $i=j+2$ and the property (iii) for $i=j+2$ imply that 
$\rho(f_{j+2}(z))\in (M_{j+1},M_{j+3})$   for $ z\in \overline S\setminus U_{j+2}$.
Therefore by \eqref{rho} and \eqref{ocenaf} for any $j\ge 0$ we obtain that 
$\rho(g(z))>M_j$ for $z\in \overline U_{j+3}\setminus U_{j+2}$,
which implies that the map $g$ is proper and that $g(z)\notin L$ 
for $z\in S\setminus U_{0}$.
Property (v) implies that 
the $k-1$ jets of $f$ and $g$ at $s$ are equal.
By the choice of $k$, $U_1$, $\epsilon_1$  the estimate  \eqref{ocenaf} implies that 
$g(\overline U_0)\cap L  =\{p\}$ 
and that $d(g(z),f(z))<\epsilon$ for $z\in K$.
This concludes the proof of the theorem.
\end{proof}

\section{Proper holomorphic curves in manifolds with the density property}

A complex manifold $X$ enjoys the \emph{density property} if the Lie algebra generated
by all $\c$-complete holomorphic vector fields is dense in the Lie algebra of all
holomorphic vector fields on $X$ (see Varolin \cite{Varolin2001JGA,Varolin2000IJM}).
Similarly, one defines the \emph{volume density property} 
of a complex manifold $X$ endowed with a holomorphic volume form $\omega$, by
considering the Lie algebra of all holomorphic vector fields on $X$ annihilating $\omega$ (see Kaliman and Kutzschebauch \cite{KalimanKutzschebauch2010IM}).
Given a compact subset $K$ in a Riemann surface $S$ 
we say that a map $f\colon K\to X$ is holomorphic 
if it is holomorphic on some neighborhood of $K$.

Part (2) in Theorem \ref{thm1} follows from the following theorem:

\begin{theorem}
\label{thmMainSec2.2}
Let $X$ be a Stein manifold of $\dim X=n>1$ with the density property or the volume density
property  equipped with a Riemannian metric $d$. 
Let $L\subset X$ be a $\Ocal (X)$-convex compact set and $p\in L$. 

Given an open Riemann surface $S$, a $\Ocal (S)$-convex compact subset $K\subset S$ with nonempty interior $\mathring K$,
$s\in \mathring K$, and a holomorphic map 
$f\colon K\to X$ with $f(s)=p$, having a finite order of contact with $L$ at $p$, and $f(K\setminus\{s\})\cap L  =\emptyset$, there exists a proper holomorphic immersion with simple double points
$g\colon S\to X$ with $g(s)=p$,  $g(S\setminus\{s\})\cap L  =\emptyset$,
and $d(g(z),f(z))<\epsilon$ for $z\in K$.
If $n\ge 3$  then $g$
can be chosen an embedding.

In particular, if $L$ is a polynomially convex compact subset in $\c^n$, $p\in L$, and
$f\colon \d\to X$ is a holomorphic map  with $f(0)=p$, having a finite order of contact with $L$ at $p$, then for any open Riemann surface $S$, $s\in S$, there is a proper holomorphic immersion with simple double points
$g\colon S\to X$ with $g(s)=p$,  $g(S\setminus\{s\})\cap L  =\emptyset$.
If $n\ge 3$  then $g$
can be chosen an embedding.
\end{theorem}

Andrist and Wold \cite{AndristWold2014AIF} proved that an open Riemann surface $S$ immerses into Stein manifold $X$ of $\dim X\ge 2$ with the (volume) density property properly holomorphically. In the case $\dim X\ge 3$ they were able to choose this immersion to be an embedding.
In \cite{AndristForstnericRitterWold2016JAM} the authors extended this result and proved that
any Stein manifold $S$ can be embedded into Stein manifold $X$ with the (volume) density property properly holomorphically if $\dim X\ge 2\dim S+1$.

Forstneri\v c and Ritter \cite{ForstnericRitter2014MZ} proved that holomorphic maps from Stein manifolds $S$ of dimension $<n$ to the complement of a compact convex set $L$  in $\c^n$ satisfy the basic Oka property with approximation and interpolation.
Furthermore, if $2\dim S\le n$ they proved that for any polynomially convex subset $L$ of $\c^n$ and for any holomorphic map $f$ defined on a $\Ocal (S)$-convex compact set $K\subset  S$ such that $f(bK)\subset \c^n\setminus L$ there exists a proper holomorphic map $g\colon S\to \c^n$
such that $g(S\setminus K)\subset  \c^n\setminus L$.
The proof depends on Anders\'en-Lempert-Forstneri\v c-Rosay theorem on approximation of 
isotopies of injective holomorphic maps by holomorphic automorphisms which also holds
on Stein manifolds with the (volume) density property. That enabled Forstneri\v c  to extend the 
above results \cite{Forstneric2018JAM}.
More precisely, in the proof of Theorem \ref{thmMainSec2.2} we use \cite[Remark 1.3]{Forstneric2018JAM}. The interpolation of a higher order jet can be easily built in the construction:

\begin{theorem}
\label{thmopen}
Let $X$ be a Stein manifold of $\dim X=n>1$ with the density property or the volume density
property  equipped with a Riemannian metric $d$. 
Let $L\subset X$ be a compact $\Ocal (X)$-convex set. 
Assume that $S$ is an open Riemann surface, $K$ is a compact $\Ocal (S)$-convex set in $S$ 
and $s\in S$.

Let $U$ be an open neighborhood of $K$ and $f\colon U\cup \{s\}\to X$ a holomorphic map such that $f(bK)\cap L=\emptyset$.
Given $\epsilon>0$ and an integer $k\ge 1$ there
exists a proper holomorphic immersion with simple double points $g\colon S\to X$ 
such that the $k-1$ jets of $f$ and $g$ at $s$ are equal,
satisfying $g(S\setminus K)\subset X\setminus L$, 
and $d(f(z),g(z))<\epsilon$ for $z\in K$.
If $n\ge 3$  then $g$
can be chosen an embedding.
\end{theorem}

\begin{proof}[Proof of Theorem \ref{thmMainSec2.2}]
We choose local coordinates in a neighborhood $V$ of $p$ and 
an open neighborhood $U'\subset U$ of $s$ in $S$ 
such that $U'$ is biholomorphic to $\d$ and $f(U')\Subset V$.
By Lemma \ref{lemmasmallperturb} there is a neighborhood $U''\Subset U'$,
$\epsilon_1\in (0,\epsilon)$, and an integer $k>0$ such that 
for any holomorphic map $ g \colon S \to X$ such that
the $k-1$ jets of $f$ and $g$ at $s$ are equal, and
$d( f ( z ), g ( z ))<\epsilon_1 $ for $ z\in U'$ it holds that $g(\overline U'')\subset V$ and
 $ g  (\overline U''\setminus \{s\})\cap L  =\emptyset$.
There is $\epsilon_2\in (0,\epsilon_1)$ such that $\dist(f(z),L)>\epsilon_2$ for $z\in K\setminus U''$.
 
By Theorem \ref{thmopen} there
exists a proper holomorphic immersion with simple double points $g\colon S\to X$ 
such that the $k-1$ jets of $f$ and $g$ at $s$ are equal,
satisfying 
$g(S\setminus K)\subset X\setminus L$, 
and $d(f(z),g(z))<\epsilon_2$ for $z\in K$.
The choice of $\epsilon_2$ implies that $g(K\setminus\{s\})\cap L =\emptyset$.
This proves the theorem.
\end{proof}

\section{Example}
\label{example}
The domain 
$$D=\{(z,w)\in\c^2\colon -\Re w + \frac 15 |w|^2 + |z|^8 +|z|^2\Re (z^6)+ 10 |z|^{10} + |wz|^2<0 \}$$
has the following properties:
\begin{itemize}
\item $D$ is a bounded pseudoconvex domain in $\c^2$ with smooth boundary, 
\item $\cd$ has a Stein Runge neighborhood basis, 
\item $0\in bD$ is a weakly pseudoconvex point of finite 1-type, 
\item $D$ is not locally convexifiable at $0$, 
\item $D$ has a global support hypersurface at $0$.
\end{itemize}

The domain $D$ is a modification of the Kohn-Nirenberg domain
$$D'=\{(z,w)\in\c^2\colon -\Re w +  |z|^8 +\frac{15}{7}|z|^2\Re (z^6)+ |wz|^2<0 \},$$
which is weakly pseudoconvex domain with smooth boundary that does not have a support hypersurface at $0$.
Calamai \cite{Calamai2014BKMS} constructed a \emph{bounded} weakly pseudoconvex domain with smooth boundary that does not have a support hypersurface at $0$.
Kol\'a\v r \cite{Kolar1995MRL,Kolar2008MZ} constructed examples of smoothly bounded 
nonconvexifiable pseudoconvex domains with local support hypersurfaces.
Our example is built on these constructions.

\begin{proof}
Let 
$$\rho(z,w)= -\Re w + \frac 15 |w|^2 + |z|^8 +|z|^2\Re (z^6)+ 10 |z|^{10} + |wz|^2, (z,w)\in\c^2.$$
It is not difficult to check that the level set $\{\rho(z,w)=0\}=bD$ is smooth.
To see that $D$ is bounded write
$\rho(z,w)= \frac 15 |w-\frac 52|^2 -\frac 54 + |z|^8(1 +\frac{\Re (z^6)}{z^6})+ 10 |z|^{10} + |wz|^2$
and we derive that
$$\{(z,w)\in\c^2\colon\rho (z,w)<0 \}\subset \Big\{(z,w)\in\c^2\colon \left|w-\frac 52\right|^2 <\frac 52 ,\ |z|<\root {10}\of{\frac 18}\Big\}.$$
We calculate the Levi form of $\rho$ 
$$\Lcal _\rho(z,w)=
\left[ \begin{array}{cc}
	16|z|^6 +7\Re (z^6)+ 250 |z|^{8} + |w|^2 & w \bar z\\
	\bar w z & \frac 15 + |z|^2\\
\end{array}\right].
$$
It is easy  to see that the Levi form is strictly positive definite on $\c^2\setminus \{0,0\}$.
Therefore $D$ is pseudoconvex, $\cd$ has a Stein Runge neighborhood basis, and
$0$ is a weakly pseudoconvex point.

To prove that $D$ is not locally convexifiable at $0$ we compute the defining equation for $bD$ near $0$: 
$$\Re w=|z|^8 +|z|^2\Re (z^6)+o(\Im w,|z|^8),$$
and by \cite[Proposition 3]{Kolar1995MRL} the domain $D$ 
is not convex in any holomorphic coordinates around $0$.
From this expression it also follows that $0$ is a point of finite type $8$.
Since $\rho(z,0)>0$ for $z\ne 0$, $\c\times \{0\}$ is a global supporting hypersurface
for $bD$ at $0$. 
\end{proof}

\subsection*{Acknowledgements}
The authors thank Mike Lawrence for bringing the question of 
the existence of proper holomorphic discs attached to boundaries of pseudoconvex domains
to their attention and Franc Forstneri\v c for helpful discussions.

The authors are supported in part by the research program P1-0291 
and grants J1-5432 and J1-9104 from ARRS, Republic of Slovenia.


{\bibliographystyle{abbrv} \bibliography{bibDS}}

\vskip 3mm

\noindent Barbara Drinovec Drnov\v sek

\noindent Faculty of Mathematics and Physics, University of Ljubljana, Jadranska 19, SI--1000 Ljubljana, Slovenia, and\\
Institute
of Mathematics, Physics and Mechanics, Jadranska 19, SI--1000 Ljubljana, Slovenia.

\noindent e-mail: {\tt barbara.drinovec@fmf.uni-lj.si}

\noindent Marko Slapar

\noindent Faculty of Education, University of Ljubljana, Kardeljeva plo\v s\v cad 16, SI--1000 Ljubljana, Slovenia, \\
Faculty of Mathematics and Physics, University of Ljubljana,
Jadranska 19, SI--1000 Ljubljana, Slovenia, and\\
 Institute
of Mathematics, Physics and Mechanics, Jadranska 19, SI--1000 Ljubljana, Slovenia.

\noindent e-mail: {\tt marko.slapar@pef.uni-lj.si}

\end{document}